\documentclass[a4paper]{amsart}
\usepackage{a4,pstricks}
\usepackage{amsthm, amsfonts, amssymb,latexsym}
\usepackage{enumerate, color}
\usepackage{psfrag,epsfig}%epsf,
\usepackage[english]{babel}
\usepackage[latin1]{inputenc}

\newtheorem{lemma}{Lemma}
\newtheorem{theorem}{Theorem}

\numberwithin{cor}{section}
\numberwithin{theorem}{section}
\numberwithin{lemma}{section}
\numberwithin{observation}{section}

\newcommand{\m}[1]{\ensuremath{\ (\text{\rm mod } #1)}}
\newcommand{\ud}{\, \mathrm{d}}
\DeclareMathOperator{\lcm}{lcm}
\DeclareMathOperator{\ord}{ord}

\def\R{\mathbb R}

\newcommand{\Abs}[1]{{\left|{#1}\right|}}

\newcommand{\p}{\mathbf p}

\newcommand{\Z}{\mathbb{Z}}

\title{On the error term of the logarithm of the lcm of a quadratic
sequence}

\author[J. Ru\'e]{Juanjo Ru\'e}
\address{J. Ru\'e: Instituto de Ciencias Matem\'aticas
(CSIC-UAM-UC3M-UCM),
28049 Madrid, Spain}
\email{juanjo.rue@icmat.es}

\author[P. \v Sarka]{Paulius \v Sarka}
\address{P. \v Sarka: Institute of Mathematics and Informatics,
Akademijos
4, Vilnius LT-08663, Lithuania and Department of Mathematics and
Informatics, Vilnius
University, Naugarduko 24, Vilnius LT-03225, Lithuania}
\email{paulius.sarka@gmail.com}

\author[A. Zumalacárregui]{Ana Zumalac\'arregui}
\address{A. Zumalacárregui: Instituto de Ciencias Matem\'aticas
(CSIC-UAM-UC3M-UCM) and Departamento de Matem\'aticas, Universidad
Aut\'onoma de Madrid, 28049 Madrid, Spain }
\email{ana.zumalacarregui@uam.es}

\begin{document}
\begin{abstract}
We study the logarithm of the least common multiple of the
sequence of integers given by $1^2+1, 2^2+1,\dots, n^2+1$. Using a result of
Homma~\cite{kosuke} on the distribution of roots of
quadratic polynomials modulo primes we calculate the error term for the asymptotics obtained by Cilleruelo~\cite{quadratic-cille}.
\end{abstract}
\maketitle
\section{Introduction}\label{sec:intro}

The first important attempt to prove the Prime Number Theorem was made
by Chebyshev. In 1853 [2] he introduced the function
$\psi(n)=\sum_{p^{m}\leq n}\log (p)$ and proved that
the Prime Number
Theorem was equivalent to the asymptotic estimate $\psi(n)\sim n$. He
also proved that if $\psi(n)/n$ had a limit as $n$ tends to infinity
then that limit is $1$. The proof of this result was only completed
(independently) two years after Chebyshev's death by Hadamard and de
la Vallée Poussin.

\smallskip

Observe that Chebyshev's function is precisely $\psi(n) =\log \lcm
\left(1,2,\dots,n \right)$, so it seems natural to consider the following question: for a given polynomial
$f(x)\in
\mathbb{Z}[x]$, what can be said about the $\log \lcm
\left(f(1),f(2),\dots,f(n) \right)$? As Hadamard and de la Vall\'ee
Poussin proved, for $f(x)=x$ this quantity asymptotically behaves as $n$. Some progress has been made in
the direction of generalising this result to a broader class of polynomials. In~\cite{linear} the authors use the Prime Number
Theorem for arithmetic progressions to get the asymptotic estimate for any
linear polynomial $f(x)=ax+b$:
$$\log \lcm \left(f(1),f(2),\dots,f(n) \right)\sim n \frac{q}{\varphi(q)}\sum_{\substack{k=1\\(k,q)=1}}^q\frac{1}{k},$$
where $q=a/\gcd(a,b)$. Recently,
Cilleruelo~\cite{quadratic-cille} extended this result to the
quadratic case, obtaining that for an irreducible polynomial $f(x)=ax^2+bx+c \in
\mathbb{Z}[x]$ the following asymptotic estimate holds:
\begin{equation}\label{cille1}
\log \lcm (f(1), f(2), \dots, f(n)) = n\log n +B_f n+o(n),
\end{equation}
where the constant $B_f$ is explicit. The author also
proves that for reducible polynomials of degree two, the asymptotic is linear in $n$.
For polynomials of higher degree nothing is known, except for products of
linear polynomials, which are studied in \cite{chinese-linear}.

\smallskip

An important ingredient in Cilleruelo's argument is a result of
Toth~\cite{toth}, a generalisation of a deep theorem of Duke, Friedlander and
Iwaniec~\cite{equidistribution-roots} about the distribution of
solutions of quadratic congruences $f(x) \equiv 0 \pmod p$, when $p$
runs over all primes. Recent improvements of the latter result in the negative
discriminant case~\cite{kosuke} allowed us to refine the method and obtain
the error term in the special case of expression \eqref{cille1}.

\smallskip

We focus our study on the particular polynomial $f(x)=x^2+1$, which
simplifies the calculation, and shows how the method developed
in~\cite{quadratic-cille} works in a clear manner.
The same ideas could be extended to general irreducible quadratic
polynomials of negative discriminant, however, a generalisation of
\cite{kosuke} (in the same direction as Toth's) would be necessary.

\smallskip

For this particular polynomial the expression for $B$ is given by
\begin{equation}\label{B:x^2+1}
\gamma-1-\frac{\log 2}{2}-\sum_{p\neq
2}\frac{\big(\frac{-1}{p}\big)\log p}{p-1}\approx -0.0662756342,
\end{equation}
where $\gamma$ is the Euler constant, $\big(\frac{-1}{p}\big)$ is
the Legendre symbol and the sum is considered over all odd prime
numbers ($B$ can be computed with high numerical precision by using
its expression in terms of L-series and
zeta-series, see~\cite{quadratic-cille} for details). More precisely,
we obtain the following estimate:
\begin{theorem}\label{theoremone}
For any $\theta<4/9$ we have
$$\log \lcm (1^2 +1, 2^2 +1, \dots, n^2 + 1) = n\log n
+Bn+O\bigg(\frac{n}{(\log n)^{\theta}}\bigg),$$
where the constant $B$ is given by Expression~\eqref{B:x^2+1}.
\end{theorem}

The infinite sum in \eqref{B:x^2+1} appears in other mathematical
contexts: as it is pointed in~\cite{Moree} this sum is closely related
to multiplicative sets whose elements are non-hypotenuse numbers
(i.e. integers which could not be written as the hypothenuse of a
right triangle with integer sides).

\subsection*{Plan of the paper:} in Section~\ref{sec:notation} we
recall the basic necessary results and fix the notation used in the
rest of the paper.  We explain the
strategy for the proof of Theorem~\ref{theoremone} in Section~\ref{sec:strategy}, which is
based on a detailed study of medium primes (Section~\ref{sec:medium}).
Then, using these partial results, in Section~\ref{sec:proof} we
provide the complete proof of Theorem~\ref{theoremone}.

\section{Background and notation}\label{sec:notation}

Throughout the paper $p$ will denote a prime number and the
Landau symbols $O$ and $o$, as well as the
Vinogradov symbols $\ll$, $\gg$ will be employed with their usual
meaning. We will also use the following notation:
\begin{align*}
\pi(n) &= \left| \{p:  p \leq n\}\right|,\\
  \pi_1(n)     &= \left| \{p: p\equiv 1\m{4}, p \leq n\}\right|, \\
  \pi_1([a,b]) &= \left| \{p: p\equiv 1\m{4}, a < p \leq b\}\right|.
\end{align*}

The Prime Number Theorem states that the following estimate holds:
\begin{align} \label{PNT}
\psi(n)&=\log\lcm
(1,2,\dots,n)=n+E(n),\,\,E(n)=O\bigg(\frac{n}{\left(\log
n\right)^{\kappa}}\bigg),\end{align}
where $\kappa$ can be chosen as large as necessary. We also use the following
estimate, which follows from Prime Number Theorem for arithmetic
progressions:
\begin{equation}\label{eq:pnt-ap}
\pi_1(n) = \frac{n}{2\log n} + O\left( \frac{n}{(\log n)^2}
\right).
\end{equation}

The result needed in order to refine the error term
of~\eqref{cille1} is the main theorem in~\cite{kosuke},
which deals with the distribution of fractional parts ${\nu/p}$,
where $p$ is a prime less than or equal
to $n$ and $\nu$ is a root in $\mathbb{Z}/p\mathbb{Z}$ of a quadratic
polynomial $f(x)$ with negative discriminant. For this $f$, we define
the discrepancy
$D_f(n)$ associated to the set of fractions $\{\nu/p:
f(\nu)\equiv 0 \m{p},\, p \leq n\}$ as
$$D_f(n)=\sup_{[u,v] \in
[0,1]}\Bigg||[u,v]|-\frac{1}{\pi(n)}\sum_{p\leq n}\sum_{\substack{u<
\nu/p\leq v \\ f(\nu) \equiv 0 \m{p}}}1\Bigg|,$$
where $|[u,v]|:= v-u$. Under these assumptions, the main theorem
of~\cite{kosuke} can be stated as follows:

\begin{theorem}\label{theoremkosuke}
  Let $f$ be any irreducible quadratic polynomial with integer
coefficients   and negative discriminant. Then for any $\delta < 8/9$
we have
  $$D_f(n)= O\left( \frac{1}{(\log n)^\delta} \right).$$
\end{theorem}

As a consequence of this result, we have the following lemma:

\begin{lemma}
  Let $g:[0,1] \to \R$ be any function of bounded variation, and $n<N$
two positive real numbers. Then for any $\delta < 8/9$
  $$\sum_{\substack{n < p < N \\ 0\leq \nu < p\\ \nu^2 \equiv -1
\m{p}}}
  g\left(\frac{\nu}{p}\right) = 2\pi_1([n,N])\int_{0}^{1}g(t) \ud t +
  O\left( \frac{N}{(\log N)^{1 + \delta}}\right).$$
  \label{integrationlemma}
\end{lemma}

\begin{proof} We know by the Koksma--Hlawka identity (see Theorem 2.11
in~\cite{niederreiter}) that for
  any sequence $A = \left\{ a_1, a_2, \dots, a_n \right\}$, $A \subset
[0,1]$, with discrepancy $D(n)$ and for any $g:[0,1] \to \R$ with
bounded variation, we have
  $$\frac{1}{n}\sum_{i=1}^{n}g(a_i) = \int_{0}^{1}g(t) \ud t +
O(D(n)),$$
  so
  \begin{align*}
    \sum_{i=n}^{N}g(a_i)
     = \sum_{i=1}^{N}g(a_i) - \sum_{i=1}^{n}g(a_i)
     = (N-n)\int_{0}^{1}g(t) \ud t + O(N D(N))+O(nD(n)).
  \end{align*}
In our case, using Theorem~\ref{theoremkosuke}, we get
  $$\sum_{\substack{n < p < N \\ 0\leq \nu < p\\ \nu^2 \equiv -1
\m{p}}}
  g\left(\frac{\nu}{p}\right) = 2\pi_1([n,N])\int_{0}^{1}g(t) \ud t +
  O\left(\frac{\pi_1(N)}{(\log N)^\delta}\right).$$
  Using rough estimate $\pi_1(N) = O\left(\frac{N}{\log N}\right)$ we
get the required
  error term.
\end{proof}

\section{The strategy}\label{sec:strategy}
The content of this section can be found in~\cite{quadratic-cille}. We include it here for completeness and to prepare the reader for the forthcoming arguments.

Denote by $P_n = \prod_{i=1}^{n}(i^2 + 1)$
and $L_n = \lcm (1^2 +1, 2^2 + 1,
\dots, n^2 + 1)$, and write $\alpha_p(n) = \ord_p(P_n)$ and
$\beta_p(n) = \ord_p(L_n)$. The argument for estimating $L_n$ stems
from the following equality:
$$\log L_n = \log P_n + \sum_{p} (\beta_p(n) - \alpha_p(n))\log
p.$$
Clearly it is not difficult to estimate $\log P_n$. Indeed, using
Stirling's approximation formula, we get
\begin{align*}
  \log \prod_{i=i}^{n}(i^2+1)
   = \log \prod_{i=1}^{n}i^2 + \log
\prod_{i=1}^{n}\left(1+\frac{1}{i^2}\right)
   = 2 \log n! + O(1) = 2n\log n - 2n + O(\log n),
\end{align*}
and so in the remainder of the paper we will be concerned with the
estimation of $\sum_{p} (\beta_p(n) - \alpha_p(n))\log p$. We start here by
making three simple observations:

\begin{lemma}\label{lemma:basic}
\mbox{}
  \begin{enumerate}[\rm i)]
    \item $\beta_2(n) - \alpha_2(n) = -n/2 + O(1)$,
    \item $\beta_p(n) - \alpha_p(n) = 0$, when $p > 2n$.
    \item $\beta_p(n) = \alpha_p(n) = 0$, when $p \equiv 3 \m{4}$.
  \end{enumerate}
\end{lemma}

\begin{proof}\mbox{}
  \begin{enumerate}[i)]
    \item $i^2 + 1$ is never divisible by $4$ and is divisible by $2$
for
      every odd $i$.
    \item Note that $\alpha_p(n) \neq \beta_p(n)$ only if there exist
$i<j
      \leq n$ such that $p|i^2+1$ and $p|j^2+1$. But this implies
      $p|(i-j)(i+j)$, and in turn $p\leq 2n$.
    \item $i^2 + 1$ is never divisible by $p \equiv 3 \m{4}$ as
      $-1$ is not a quadratic residue modulo such prime.
  \end{enumerate}
\end{proof}

Since we have dealt with the prime $2$, from now on we will only consider
odd primes. Lemma~\ref{lemma:basic} also states that it is sufficient to
study the order of prime numbers which are smaller than $2n$ and are
equivalent to $1$ modulo $4$. We split these primes in two groups: ones
that are smaller than $n^{2/3}$ and others that are between $n^{2/3}$ and
$2n$, \emph{small} and \emph{medium} primes respectively.

The computation for small primes is easy and it is done in the lemma
below, after obtaining simple estimates for $\alpha_p(n)$ and $\beta_p(n)$.
Analysis of medium primes, which is left for the next section, is more
subtle and will lead to improvement of the error term.

\begin{lemma} For primes $p \equiv 1 \m{4}$ the following estimates
hold:
  \begin{enumerate}[\rm i)]
    \item $\beta_p(n) \ll \frac{\log n}{\log p}$,
    \item $\alpha_p(n) = \frac{2n}{p-1} + O\left( \frac{\log n}{\log
p} \right)$.
  \end{enumerate}
  \label{alphabetaestimates}
\end{lemma}

\begin{proof}\mbox{}
  \begin{enumerate}[i)]
    \item  It is clear that $\beta_{p}(n)$ satisfies $p^{\beta_{p}(n)}
      \leq n^2 +1$, so $$\beta_{p}(n) \leq \frac{\log (n^2 + 1)}{\log
p} \ll
      \frac{\log n}{\log p}.$$

    \item In order to estimate $\alpha_{p}(n)$ note that for primes $p
      \equiv 1 \m{4}$ equation $i^2 \equiv -1 \m{p^a}$ has two
solutions
      $\nu_1$ and $\nu_2$ in the interval $[1,p^a]$ and every other
      solution is of the form $\nu_1 + kp^a$ or $\nu_2 + kp^a$,
$k\in
      \Z$. The number of times $p^a$ divides $i^2+1$, $i=1,\dots,n$ is given by
      \begin{equation}
	2+\left\lfloor\frac{n-\nu_1}{p^a} \right\rfloor +
\left\lfloor\frac{n-\nu_2}{p^a}\right\rfloor,
	\label{alphaestimate}
      \end{equation}
      which equals to $0$ for $p^a > n^2 + 1$ and $2n/p^a +O(1)$
for $p^a \leq n^2 +1$. Therefore we get
      \begin{align*}
	\alpha_{p}(n)
	& = 2\sum_{j=1}^{\big\lfloor \frac{\log (n^2 + 1)}{\log
p}\big\rfloor }
	\frac{n}{p^j} + O\left(\frac{\log n}{\log p}\right)\\
	& = 2n\sum_{j=1}^{\infty} \frac{1}{p^j} - 2n \sum_{j =
	\big\lfloor\frac{\log (n^2 + 1)}{\log p}\big\rfloor +
1}^{\infty}  \frac{1}{p^j} +
	O\left(\frac{\log n}{\log p}\right)\\
	& = \frac{2n}{p-1} + O\left(\frac{\log n}{\log p}\right),
      \end{align*}
  \end{enumerate}
and the claim follows.
\end{proof}

\begin{lemma} The following estimate holds:
  $$\sum_{2<p< n^{2/3}} (\alpha_p(n) - \beta_p(n))\log p = \sum_{2<p<
n^{2/3}}
  \frac{\left(1+\big(\frac{-1}{p}\big)\right)n\log p}{p-1} +
O(n^{2/3}).$$
  \label{smallprimeslemma}
\end{lemma}

\begin{proof}
  Using the estimates from Lemma~\ref{alphabetaestimates} we get
  $$\sum_{2<p< n^{2/3}} \beta_{p}(n) \log p \ll \sum_{2<p < n^{2/3}}
\log n \ll n^{2/3},$$
  and also
  \begin{align*}
    \sum_{2<p< n^{2/3}} \alpha_{p}(n)\log p\,
    & = \sum_{\substack{p< n^{2/3} \\ p \equiv 1 \m{4}}}
\left(\frac{2n\log
    p}{p-1}
    + O(\log n)\right)\\
    & = \sum_{2<p< n^{2/3}}
\frac{\left(1+\big(\frac{-1}{p}\big)\right)n\log p}{p-1} + O(n^{2/3}),
  \end{align*}
  and hence the claim follows.
\end{proof}

\section{Medium primes}\label{sec:medium}

In order to deal with the remaining primes, we note, that if prime
$p\equiv
1 \m{4}$ is in the range $n^{2/3} \leq p \leq 2n$ then it divides $i^2
+ 1$
for some $i\leq n$. However, since such a prime is sufficiently large
compared to $n^2 + 1$, the case that $p^2$ divides some $i^2 +1$,
$i\leq n$
is unlikely.

Having this in mind, we separate contribution of higher
degrees from the contribution of degree~$1$. Define for $p \equiv 1
\m{4}$:
\begin{align*}
  \alpha^*_{p}(n) & = \left|\{i: \, p|i^2+1, \, i\leq n \}\right|, \\
  \beta^*_{p}(n) & = 1,
\end{align*}
and, for $p\equiv 3\m{4}$, $\alpha^{*}_{p}(n) = \beta^{*}_{p}(n) = 0$.
Then
\begin{align}\label{3-sum}
  \sum_{n^{2/3} \leq p \leq 2n} (\beta_{p}(n) - \alpha_{p}(n))\log p
   = &\sum_{n^{2/3} \leq p \leq 2n} (\beta_{p}(n) - \beta^{*}_{p}(n)
   - \alpha_{p}(n) + \alpha^{*}_{p}(n))\log p\,+ \\
  & \sum_{n^{2/3} \leq p \leq 2n} \beta^{*}_{p}(n) \log p -
\sum_{n^{2/3} \leq p \leq 2n} \alpha^{*}_{p}(n) \log p.\nonumber
\end{align}

We now estimate each sum in the previous equation. We start estimating
the first one:

\begin{lemma}The following estimate holds:
 $$\sum_{n^{2/3} \leq p \leq 2n} (\beta_{p}(n) - \beta^{*}_{p}(n)
  - \alpha_{p}(n) + \alpha^{*}_{p}(n))\log p \ll n^{2/3}\log n.$$
  \label{highdegreeslemma}
\end{lemma}

To prove this lemma we need some preliminary results. As it was
intended, $(\beta_{p}(n) - \beta^{*}_{p}(n) -
\alpha_{p}(n) + \alpha^{*}_{p}(n))\log p$ is nonzero only when
$p^2|i^2 +
1$ for some $i \leq n$. We claim, that number of such primes is small:

\begin{lemma}The following estimate holds:
  $$\left| \{p: \, p^2|i^2 +1, \, n^{2/3} \leq p \leq 2n, i \leq n\}\right|
\ll
  n^{2/3}.$$
  \label{badprimeslemma}
\end{lemma}

\begin{proof}
 Let us split the interval $[n^{2/3},2n]$ into dyadic intervals,
consider one
 of them, say $[Q,2Q]$, and define
 $$P_k=\{\,p:\, i^2+1=kp^2\ \text{ for some }i\leq n\}.$$
 We estimate the size of the set $P_k\cap [Q,2Q]$, which is nonempty
only
 when $k\leq (n^2+1)/Q^2$. For every $p\in P_k\cap [Q,2Q]$ we have
 $i^2-kp^2=(i+\sqrt{k}p)(i-\sqrt{k}p)=-1$, thus
 $$\Abs{\frac{i}{p}-\sqrt{k}} = \frac{1}{p^2}\left( \frac{i}{p} +
\sqrt{k}
 \right)^{-1} \leq \frac{1}{p^2}\leq\frac{1}{Q^2}.$$
 On the other hand, all fractions $i/p$, $p\in P_k$, are pairwise
 different, since $ip'=i'p$ implies $p=p'$ (otherwise $p|i$, and so
 $p|i^2-kp^2=-1$), therefore
 $$\Abs{\frac{i}{p}-\frac{i'}{p'}}\geq
\frac{1}{pp'}\gg\frac{1}{Q^2}.$$
 Combining both inequalities we get $\Abs{P_k\cap[Q,2Q]}\ll 1$ for every
 $k\leq (n^2+1)/Q^2$. Recalling that $P_k\cap[Q,2Q]$ is empty for
 other values of $k$ we have
 $$\left| \{p: \, p^2|i^2 +1, \, Q \leq p \leq 2Q, \, i \leq n\}\right| =
\Abs{\cup_k(P_k \cap
 [Q,2Q])}\ll
  \frac{n^2}{Q^2}.$$
 Summing over all dyadic intervals the result follows.
\end{proof}

Now we use this estimate to prove Lemma~\ref{highdegreeslemma}.
\begin{proof}[Proof of Lemma~\ref{highdegreeslemma}]
  We use estimates from Lemma~\ref{alphabetaestimates} and the
estimate for $\alpha_p^*(n)$, which follows from Expression~\eqref{alphaestimate}:
  \begin{align*}
    \beta_p(n)   &\ll \frac{\log n}{\log p},\\
    \alpha_p(n)  &= \frac{2n}{p-1}+O\left(\frac{\log n}{\log
p}\right),\\
    \alpha_p^*(n)&= \frac{2n}{p}+O(1).
  \end{align*}
  For any prime $n^{2/3}<p<2n$, such that $p^2| i^2+1$ for some $i\leq
n$,
  we get
  $$\Abs{\beta_{p}(n) - \beta^{*}_{p}(n) - \alpha_{p}(n) +
  \alpha^{*}_{p}(n)}=\frac{2n}{p(p-1)}+ O\left(\frac{\log n}{\log
p}\right)\ll \frac{\log n}{\log p}.$$
  It follows from lemma~\ref{badprimeslemma} that the number of such
primes is
  $\ll n^{2/3}$, thus
  $$\sum_{n^{2/3} \leq p \leq 2n} (\beta_{p}(n) - \beta^{*}_{p}(n) -
  \alpha_{p}(n) + \alpha^{*}_{p}(n))\log p\ll n^{2/3} \log n.$$
\end{proof}

We continue estimating the second sum in Equation~\eqref{3-sum}:

\begin{lemma} The following estimate holds:
  $$\sum_{n^{2/3} \leq p \leq 2n} \beta^{*}_{p}(n) \log p = n +
O\left(\frac{n}{\log n}\right).$$
  \label{betastarlemma}
\end{lemma}

\begin{proof}
  Summing by parts and using estimate~\eqref{eq:pnt-ap} for $\pi_1(x)$
we get:
  \begin{align*}
    \sum_{n^{2/3} \leq p \leq 2n} \beta^{*}_{p}(n) \log p
    & = \sum_{\substack{n^{2/3} \leq p \leq 2n \\ p \equiv 1 \m{4}}}
\log p\\
    & = \sum_{\substack{p \leq 2n \\ p \equiv 1 \m{4}}} \log p +
O(n^{2/3})\\
    & = \log{(2n)}\, \pi_1(2n) - \int_{2}^{2n} \frac{\pi_1(t)}{t} \ud
t +
    O(n^{2/3})\\
    & = n + O\left(\frac{n}{\log n}\right).
  \end{align*}
\end{proof}

Finally, we deal with the contribution of the coefficients
$\alpha^{*}_{p}$. In this point we need to take care of the error term
in a more detailed way:

\begin{lemma}
  For any $\delta < 8/9$ the following estimate holds:
  $$\sum_{n^{2/3} \leq p \leq 2n} \alpha^{*}_{p}(n) \log p =
  \sum_{n^{2/3}\leq p \leq 2n}
\frac{\left(1+\big(\frac{-1}{p}\big)\right)n\log
  p}{p-1} +O\left(\frac{n}{(\log n)^{\delta/2}}\right).$$
  \label{alphastarlemma}
\end{lemma}

\begin{proof}
  Using~(\ref{alphaestimate}) and noting that $\nu_1 +\nu_2 = p$,
where
  $1\leq \nu_1,\nu_2 \leq p$ are solutions of $i^2 \equiv -1 \m{p}$,
we get
\begin{align*}
  \alpha^{*}_p(n)
  & = 2 + \left\lfloor \frac{n - \nu_1}{p}\right\rfloor + \left\lfloor
\frac{n - \nu_2}{p}
  \right\rfloor \\
  & = 2 + \frac{2n}{p} - \frac{\nu_1 + \nu_2}{p} - \left\{ \frac{n -
\nu_1}{p}
  \right\} - \left\{ \frac{n - \nu_2}{p} \right\} \\
  & = \frac{2n}{p} + \frac{1}{2}  - \left\{ \frac{n - \nu_1}{p}
  \right\} + \frac{1}{2} - \left\{ \frac{n - \nu_2}{p} \right\},
\end{align*}
so the sum over all primes in the interval $[n^{2/3}, 2n]$ is equal to
\begin{align*}
  \sum_{n^{2/3}\leq p \leq 2n} \hspace{-10pt}\alpha^{*}_p(n) \log{p}
 = \hspace{-11pt}\sum_{n^{2/3} \leq p \leq 2n} \frac{\left(1 +
\big(\frac{-1}{p}\big)\right)n \log{p}}{p}
 + \hspace{-11pt}\sum_{\substack{n^{2/3} \leq p \leq 2n \\ \nu^2
\equiv -1 \m{p} \\ 0
\leq \nu < p}} \hspace{-10pt}\log{p} \left(\frac{1}{2} - \left\{
\frac{n - \nu}{p}
\right\}\right).
\end{align*}
We rewrite
$$ \sum_{n^{2/3} \leq p \leq 2n} \frac{\left(1 +
\big(\frac{-1}{p}\big)\right)n \log{p}}{p} =
\sum_{n^{2/3} \leq p \leq 2n} \frac{\left(1 +
\big(\frac{-1}{p}\big)\right)n \log{p}}{p-1}
+O(n^{1/3}\log n)$$
and
$$\sum_{\substack{n^{2/3} \leq p \leq 2n \\ \nu^2 \equiv -1 \m{p} \\ 0
\leq
\nu < p}} \hspace{-10pt}\log{p} \left(\frac{1}{2} - \left\{ \frac{n -
\nu}{p} \right\}\right)
= \log{n}\hspace{-10pt}\sum_{\substack{0 < \nu < p \leq 2n \\ \nu^2
\equiv
-1 \m{p}}} \hspace{-10pt}\left(\frac{1}{2} - \left\{ \frac{n - \nu}{p}
\right\}\right) + O\left(\frac{n}{\log{n}}\right).$$
Notice that for any sequence $a_p$ satisfying $a_p \ll 1$ we have by a
summing by
parts argument that
$$ \sum_{p<x} a_p \log{p} = \log{x} \sum_{p<x}a_p -
\int_{1}^{x}\frac{1}{t} \sum_{p < t} a_p \ud t = \log{x} \sum_{p<x}a_p
+
O\left(\frac{x}{\log x}\right).$$
In order to get the claimed bound, it remains to show that
$$\sum_{\substack{0 < \nu < p \leq 2n \\ \nu^2 \equiv -1 \m{p}}}
\hspace{-10pt}\left(\frac{1}{2} - \left\{ \frac{n - \nu}{p}
\right\}\right)
= O\left(\frac{n}{(\log n)^{1+\delta/2}}\right).$$
To do that, we divide the summation interval into $1+H$ parts $[1,2n]
=
[1,A]\cup L_1 \cup \dots \cup L_H$, where
$$L_i = \left(\frac{2nAH}{2n(H-i+1) + A(i-1)},
\frac{2nAH}{2n(H-i)+Ai}\right]
.$$
We choose $A=\lfloor n/(\log n)^{\delta/2}\rfloor$ and $H=\lfloor
(\log n)^{\delta}\rfloor$ in order to
minimize the error term, but we continue using these notations for the
sake of conciseness.

Observe that in every of these parts, except the first one, $n/p$ is
almost constant, which enables to use the fact that $\nu/p$ is well
distributed. More precisely, if $p \in L_i$ then
$$\frac{n}{p} \in [\lambda_i, \lambda_{i-1}) :=
\left[\frac{2n(H-i) + Ai}{2AH}, \frac{2n(H-i+1) +
A(i-1)}{2AH}\right),$$
and the length of such interval is small: $|[\lambda_i,
\lambda_{i-1})| = \frac{2n - A}{2AH}.$ We would then like to replace
$\frac{n}{p}$ by $\lambda_i$
whenever $\frac{n}{p}\in [\lambda_i, \lambda_{i-1})$ using
\begin{equation}
  \label{constantisation}
  \left\{ \frac{n}{p} - \frac{\nu}{p} \right\} = \left\{\lambda_i -
\frac{\nu}{p} \right\} + \left\{\frac{n}{p} - \lambda_i\right\},
\end{equation}
but this equality does not hold if $\lambda_i < \frac{\nu}{p} + k <
\frac{n}{p}$ for some integer $k$. We extend this
interval to $\lambda_i \leq \frac{\nu}{p} +k \leq \lambda_{i-1}$,
rewrite it as $\frac{\nu}{p}
\in [\lambda_i, \lambda_{i-1}]_1$ and split the previous sum into
three parts:
\begin{equation*}
  \sum_{\substack{0 < \nu < p \leq 2n \\ \nu^2 \equiv -1 \m{p}}}
  \left(\frac{1}{2} - \left\{ \frac{n-\nu}{p} \right\}\right) =
\Sigma_1 + \Sigma_2 +
  \Sigma_3 + O(\pi_1(A)),
\end{equation*}
where $\Sigma_1,\,\Sigma_2$ and $\Sigma_3$ are defined as
\begin{align*}
  \Sigma_1
  & = \sum_{i=1}^{H} \sum_{\substack{0 \leq \nu
  < p \in L_i \\ \nu^2 \equiv -1 \m{p}}} \left(\frac{1}{2} -
  \left\{\lambda_i
  - \frac{\nu}{p} \right\} \right), \\
  \Sigma_2
  & = \sum_{i=1}^{H} \sum_{\substack{0 \leq \nu
  < p \in L_i\\ \nu^2 \equiv -1 \m{p} \\ \frac{\nu}{p} \not \in
  [\lambda_i, \lambda_{i-1}]_1}} \left(\left\{ \lambda_i -
  \frac{\nu}{p} \right\} - \left\{ \frac{n}{p} - \frac{\nu}{p}
  \right\}  \right),\\
  \Sigma_3
  & = \sum_{i=1}^{H} \sum_{\substack{0 \leq \nu < p \in L_i\\ \nu^2
\equiv
  -1 \m{p} \\ \frac{\nu}{p} \in [\lambda_i, \lambda_{i-1}]_1 }}
\left(\left\{ \lambda_i -
  \frac{\nu}{p} \right\} - \left\{ \frac{n}{p} - \frac{\nu}{p}
\right\}
  \right).
\end{align*}

Recall that $A = n/(\log n)^{\delta/2} +O(1)$ and $H = (\log
n)^{\delta} +
O(1)$, so $\pi_1(A) = O(n/(\log n)^{1+\delta/2})$. We now estimate
each
of the sums $\Sigma_1, \Sigma_2, \Sigma_3$ separately, making use of
Lemma~\ref{integrationlemma}. For the first one note that
$$\int_{0}^{1} \left(\frac{1}{2} - \left\{ \lambda_i - t \right\}
\right)
\ud t= 0,$$
so we get
\begin{align*}
  \Sigma_1
  & = \sum_{i=1}^{H} \sum_{\substack{0 \leq \nu
  < p \in L_i \\ \nu^2 \equiv -1 \m{p}}} \left( \frac{1}{2} - \left\{
\lambda_i -
  \frac{\nu}{p} \right\} \right) \\
  & = \sum_{i=1}^{H} \; O\left(\frac{2nAH}{2n(H-i) + Ai} \bigg/
\left(\log
  \frac{2nAH}{2n(H-i) + Ai}\right)^{1+\delta} \right) \tag{$*$}\\
  & = O\left( \frac{2nAH}{(\log n)^{1+\delta}} \int_0^{H}
  \frac{\ud i}{2n(H-i) + Ai}\right) \\
  & = O\left( \frac{2nAH}{(\log n)^{1+\delta}} \frac{\log 2n/A}{2n-A}
  \right)  = O\left( \frac{n \log \log n}{(\log n)^{1+\delta/2}}
\right).
\end{align*}
For the second sum we use Equation~\eqref{constantisation}:
\begin{align*}
  \Sigma_2
  & = \sum_{i=1}^{H} \sum_{\substack{0 \leq \nu
  < p \in L_i\\ \nu^2 \equiv -1 \m{p} \\ \frac{\nu}{p} \not \in
  [\lambda_i, \lambda_{i-1}]_1}} \left\{ \frac{n}{p} -
  \lambda_i \right\} \\
  & \leq \sum_{i=1}^{H} \sum_{\substack{0 \leq \nu
  < p \in L_i\\ \nu^2 \equiv -1 \m{p}}} |[\lambda_i, \lambda_{i-1}]|
  \tag{$**$}\\
  & \leq \frac{2n - A}{2AH} 2\pi_1(2n)= O \left( \frac{n}{(\log n)^{1
+ \delta/2}} \right).
\end{align*}
Finally, for the third sum we use the notation
$\mathbb{I}_{[\lambda_i,
\lambda_{i-1}]_1}$ for the indicator function of the interval
$[\lambda_i,
\lambda_{i-1}]$ modulo $1$, which satisfies
$$\int_{0}^{1} \mathbb{I}_{[\lambda_i, \lambda_{i-1}]_1 }(t) \ud t =
|[\lambda_i, \lambda_{i-1}]|,$$
so using Lemma~\ref{integrationlemma} we get
\begin{align*}
  \Sigma_3
  & \ll \sum_{i=1}^{H} \sum_{\substack{0 \leq \nu
  < p \in L_i\\ \nu^2 \equiv -1 \m{p} \\ \frac{\nu}{p} \in [\lambda_i,
\lambda_{i-1}]_1 }} 1 = \sum_{i=1}^{H} \sum_{\substack{0 \leq \nu
  < p \in L_i\\ \nu^2 \equiv -1 \m{p} }} \mathbb{I}_{[\lambda_i,
\lambda_{i-1}]_1}\left(\frac{\nu}{p}\right) \\
  & = \sum_{i=1}^{H} 2\pi_i(L_i)|[\lambda_i, \lambda_{i-1}]| +
  O\left(\frac{2nAH}{2n(H-i) + Ai} \bigg/ \left(\log
\frac{2nAH}{2n(H-i) +
  Ai}\right)^{1+\delta} \right). \\
  & = O\left( \frac{n \log \log n}{(\log n)^{1+\delta/2}} \right).
  \qquad{\text{(continuing as from ($*$) and ($**$))}}
\end{align*}
Finally, we note that any function $f$ satisfying $f(n) = O\left(
\frac{n
\log \log n}{(\log n)^{1+\delta/2}}\right)$ for every $\delta < 8/9$
also satisfies
$f(n)= O\left( \frac{n}{(\log n)^{1+\delta/2}}\right)$ for every
$\delta < 8/9$, hence this concludes the proof.
\end{proof}

\section{Proof of theorem~\ref{theoremone}}\label{sec:proof}

Combining results from Lemmas~\ref{smallprimeslemma},
\ref{highdegreeslemma}, \ref{betastarlemma} and \ref{alphastarlemma}, and taking $\theta=\delta/2$,
we get
that for any constant $\theta<4/9$
\begin{equation}
  \log L_n =2n\log n -n \left(1+\frac{\log 2}{2}+\sum_{2< p \leq 2n}
  \frac{\big(1+\big(\frac{-1}{p}\big)\big)\log p}{p-1}\right)
+O\left(\frac{n}{(\log
  n)^{\theta}}\right).
  \label{final}
\end{equation}
Note that,
$$\sum_{2< p \leq 2n} \frac{\big(1+\big(\frac{-1}{p}\big)\big)\log
p}{p-1}=\sum_{2< p \leq 2n} \frac{\log p}{p-1}+\sum_{2< p \leq 2n}
\frac{\big(\frac{-1}{p}\big)\log p}{p-1},$$
where the second sum is bounded, since the sum over all primes is
convergent. We finish the proof by estimating both sums in two
following
lemmas:
\begin{lemma} \label{last_sum_lemma}
The following estimate holds:
$$\sum_{ 2 < p \leq 2n} \frac{\log p}{p-1}= \log n - \gamma
+O\left(\frac{1}{\log n}
\right).$$
\end{lemma}
\begin{proof}
This estimate is well-known. However, for completeness, we include a
detailed proof. Write
\begin{align*}
  \sum_{p \leq x} \frac{\log p}{p(1-\frac{1}{p})}
  &=\sum_{p \leq x} \frac{\log
p}{p}\left(1+\frac{1}{p}+\frac{1}{p^2}+\cdots\right)\\
  &=\sum_{p^j \leq x} \frac{\log p}{p^j}+\sum_{\substack{p^j > x \\ p
\leq
  x}}\frac{\log p}{p^j}.
\end{align*}
By Merten's theorem
$$\sum_{p^j \leq x} \frac{\log p}{p^j} = \log x -\gamma +o(1),$$
and the error term can be improved using Prime Number Theorem in the
form
(\ref{PNT}) and summation by parts:
\begin{align*}
\sum_{p^j \leq x} \frac{\log p}{p^j}
  &=\frac{\psi(x)}{x}+\int_2^x\frac{\psi(t)}{t^2}\ud t \\
  &=1 + O\left( \frac{1}{\log x} \right) +
\int_2^x\frac{1}{t}dt+\int_2^x\frac{E(t)}{t^2}dt\\
  &=1 + O\left(  \frac{1}{\log x} \right) + \log x - \log 2 +
\int_2^{\infty}\frac{E(t)}{t^2}dt -
\int_x^{\infty}\frac{E(t)}{t^2}dt\\
  &=\log x - \gamma + O\left(  \frac{1}{\log x} \right) .
\end{align*}
For the second term we have
\begin{align*}
  \sum_{\substack{p^j > x \\ p \leq x}}\frac{\log p}{p^j}
  & = \sum_{p\leq x}\frac{\log p}{p^{\left\lfloor\frac{\log x}{\log p}
  \right\rfloor }\cdot
p}\left(1+\frac{1}{p}+\frac{1}{p^{2}}+\cdots\right)\\
  & \leq \frac{1}{x^{1/2}}\sum_{p\leq x}\frac{\log
p}{p\left(1-\frac{1}{p}\right)}\\
  & =O\bigg(\frac{\log x}{x^{1/2}}\bigg).
\end{align*}
In our case we get
$$\sum_{ 2 < p \leq 2n} \frac{\log p}{p-1}= \sum_{ p \leq 2n}
\frac{\log p}{p-1} -\log 2 = \log n - \gamma +O\left(\frac{1}{\log n}
\right).$$
\end{proof}

\subsection*{\textbf Remark:} The error term in the previous lemma can
be sharpened to $O(\exp(-c\sqrt{\log n}))$ (where $c$ is a constant)
using the estimate
$$\sum_{p^j \leq x} \frac{\log p}{p^j} = \log x - \gamma +O(\exp(-c\sqrt{\log x}))$$
which can be found, for instance, in~\cite{mult-num-theory} (Exercise $4$, page $182$).

\begin{lemma}
  The following estimate holds:
  $$\sum_{2< p \leq 2n} \frac{\big(\frac{-1}{p}\big)\log
  p}{p-1}=\sum_{p\neq 2} \frac{\big(\frac{-1}{p}\big)\log p}{p-1} +
  O\left(\frac{1}{ \log n} \right).$$
  \label{last_sum_lemma_2}
\end{lemma}

\begin{proof}
  We know that
$$\sum_{2< p \leq 2n} \frac{\big(\frac{-1}{p}\big)\log p}{p-1}=\sum_{p
\neq
2} \frac{\big(\frac{-1}{p}\big)\log p}{p-1}+o(1),$$
since the sum over all primes is convergent. Alternatively, we can
express the sum as follows
$$\sum_{2< p \leq 2n} \frac{\big(\frac{-1}{p}\big)\log
p}{p-1}=\sum_{\substack{2< p \leq 2n\\p\equiv 1\bmod{4}}} \frac{\log
p}{p-1}-\sum_{\substack{2< p \leq 2n\\ p\equiv 3\bmod{4}}} \frac{\log
p}{p-1}.$$
It follows from Prime Number Theorem in arithmetic progressions that
$$\psi_1(x) :=\sum_{\substack{p^j < x\\ p\equiv 1\bmod 4}}\log
p=\frac{x}{2}+E_1(x),\,\, E_1(x)=O\left(\frac{x}{\left(\log
x\right)^{2}}\right).$$
(equivalently for $\psi_3(x)$, where the summation is over primes
equivalent to $3$ modulo 4). We
can use this as in proof of Lemma~\ref{last_sum_lemma} to get

\begin{align*}
\sum_{\substack{2< p \leq 2n\\p\equiv 1\bmod{4}}} \frac{\log p}{p-1}&=
\frac{1}{2}\log n+C_1+ O\left( \frac{1}{\log  n} \right),\\
\sum_{\substack{2< p \leq 2n\\ p\equiv 3\bmod{4}}} \frac{\log
p}{p-1}&=\frac{1}{2}\log n+C_3 + O\left( \frac{1}{\log  n} \right).
\end{align*}
The difference $C_1 - C_3$ then has to be equal to $\sum_{p\neq 2}
\frac{\big(\frac{-1}{p}\big)\log p}{p-1}$ and we get the required
convergence rate.
\end{proof}

\subsection*{\textbf Acknowledgments:} This work was done during second
author's visit at Universidad Aut\'{o}noma de Madrid in Winter of
2011. He would like to thank people of Mathematics Department and
especially Javier Cilleruelo for their warm hospitality.
The authors are also grateful for his advice and helpful
suggestions in the preparation of this paper.

\smallskip

Pieter Moree is greatly thanked for noticing the connection between the
constant~\eqref{B:x^2+1} and the non-hypotenuse numbers, for pointing reference~\cite{mult-num-theory} in relation to Lemma~\ref{last_sum_lemma}
and also for useful comments.

\smallskip

The first author is supported by a JAE-DOC grant from the JAE program
in CSIC, Spain. The last author is supported by Departamento de
Matem\'aticas of Universidad Autónoma de Madrid, Spain.

\small
\bibliography{lcm}
\bibliographystyle{abbrv}

\end{document}